\documentclass[12pt]{amsart}
\ProvidesLanguage{portuges}
\usepackage[latin1]{inputenc}
\usepackage{indentfirst}
\usepackage{amsfonts}
\usepackage{graphicx}
\usepackage{amsmath}
\usepackage{amssymb}
\usepackage{latexsym}
\usepackage{cancel}
\usepackage{amscd}
\usepackage[all]{xy}
\usepackage{tikz-cd}
\usepackage{enumerate}
\newcommand{\aut}{\operatorname{Aut}}
\newcommand{\F}{\mathbb {F}}

\newtheorem{theorem}{Theorem}[section]
\newtheorem{definition}[theorem]{Definition}
\newtheorem{lemma}[theorem]{Lemma}
\newtheorem{corollary}[theorem]{Corollary}
\newtheorem{proposition}[theorem]{Proposition}
\newtheorem{remark}[theorem]{Remark}

\begin{document}

\title[Weierstrass semigroup and Automorphism group of the curves $\mathcal{X}_{n,r}$]{Weierstrass semigroup and Automorphism group of the  curves $\mathcal{X}_{n,r}$}

\author[H.\, Borges]{H.\, Borges}
\author[A.\, Sep\'ulveda]{A.\, Sep\'ulveda}
\author[G.\, Tizziotti]{G.\, Tizziotti}

\maketitle

\textbf{Abstract} In this paper, we determine the Weierstrass semigroup $H(P_{\infty})$ and the full automorphism group of a certain family of curves $\mathcal{X}_{n,r}$, which was recently   introduced by Borges and Concei\c{c}\~ao.

\section{Introduction}

Since the algebraic geometric (AG)  codes were introduced by V.D. Goppa (\cite{goppa1}, \cite{goppa2}), the study of several topics in the  theory of algebraic curves has been intensified. A particular case is the investigation  of   Weierstrass semigroups $H(P)$ at  points $P$ of an algebraic  curve. It is well known that the structure of $H(P)$ provides both local and global   information regarding this curve.  The knowledge of such structure can prove quite useful, for instance, in improving the lower bound on the mimimum distance of AG codes arising from the  curve (see \cite{carvalho2}, \cite{garcia}, \cite{kirfel}, \cite{gretchen}). 

The  group of automorphisms of a curve is another topic  that plays a prominent role in several  aspects of the theory. To be concise, it should be noted  that many curves have a characterization based on the structure and the action of their automorphism groups.  It is worth mentioning that AG codes can also be studied in connection with these groups (see \cite{stichtenoth},\cite{wesemeyer}). Applications of the automorphism groups in the systematics of encoding and decoding can be found  in \cite{heegard} and \cite{joyner}, respectively.

In \cite{c2}, Borges and Concei\c{c}\~ao used minimal value set polynomials to introduce a large family of curves that generalize the Hermitian curve. An example  is the following curve over $\F_{q^n}$:
  $$
{\mathcal H}: \hspace{.3 cm}{\rm T_n}(y)={\rm T_n}(x^{q^r+1})\,\quad \pmod{x^{q^{n}}-x}
   $$
where, for a symbol $z$, 
   $$
{\rm T_n}(z):=z^{q^{n-1}}+z^{q^{n-2}}+\ldots+z\,,
   $$ 
and $r=r(n)\geq n/2$ is the smallest positive integer such that $\gcd(n,r)=1$, i.e.,

   \begin{equation}\label{r1}
r= \begin{cases} 
1, & \text{ if }\quad n=2\\
n/2+1, & \text{ if }\quad n\equiv0\pmod4\\
n/2+2, & \text{ if }\quad n\geq 6,\, n\equiv2\pmod4\\
(n+1)/2, & \text{ if }\quad n \text{ is odd}\, .
\end{cases} 
  \end{equation}
Borges and Concei\c{c}\~ao have shown  that the curve ${\mathcal H}$ has a somewhat large ratio $${\#\mathcal H}(\mathbb{F}_{q^n})/g({\mathcal H}),$$
where   ${\mathcal H}(\F_{q^n})$ is the set of $\mathbb{F}_{q^n}$-rationals points, and  $g({\mathcal H})$ is the genus of  $\mathcal{H}$. They also computed the Weierstrass semigroup $H(P_{\infty})$, where $P_{\infty}=(0:1:0)$ is the only singular point on the projective closure of  $\mathcal{H}$.  In particular, is was proved   that   ${\mathcal H}$ is an example of a so-called Castle curve,  which are 
special types of curves potentially suitable for construction of AG codes with good parameters  (see \cite{MST}).

   In this paper, we further investigate the family of curves introduced in  \cite{c2}. More precisely, within this family,    we consider a set of curves $\mathcal{X}_{n,r}$ (which includes the curve $\mathcal{H}$)  and determine  the Weierstrass semigroup $H(P_{\infty})$ as well as the curves'  full group of automorphisms. In particular, we arrive at a larger set of Castle curves.

This paper is organized as follows. In Section 2, some basics concepts related to Weierstrass semigroups  and automorphism groups of curves are recalled, and some  properties of the curves $\mathcal{X}_{n,r}$ are presented. In the Section 3, we determine the Weierstrass semigroup $H(P_{\infty})$ and prove that $\mathcal{X}_{n,r}$ are Castle curves, while providing an alternative, somewhat simpler, plane model for the curves $\mathcal{X}_{n,r}$.  Finally, in  Section 4, the full automorphism group of $\mathcal{X}_{n,r}$ is obtained.

\section{Preliminaries}

Hereafter  we  use the following notation: 

\begin{itemize}

\item $\mathbb{F}_{q}$ is the finite field with $q$ elements.
\item $\mathcal{X}$ is a projective, non-singular, geometrically irreducible algebraic curve of genus $g\geq 2$ defined over $\mathbb{F}_{q}$.
\item $\mathbb{F}_{q}(\mathcal{X})$ denotes  the field of $\mathbb{F}_{q}$-rational functions on $\mathcal{X}$.
\item For a  point $P\in \mathcal{X}$, the discrete valuation at $P$ is denoted by $v_P$.
\item For $f \in \mathbb{F}_{q}(\mathcal{X})\backslash \{0\}$, the pole divisor of $f$ is denoted by $div_{\infty}(f)$.
\item $\# \mathcal{X}(\mathbb{F}_q)$  denotes the number of $\mathbb{F}_{q}$-rational points of $\mathcal{X}$.
\item $\mathbb{N}_{0}=\{0,1,2,\ldots\}$ denotes  the semigroup  of non-negative integers.

\end{itemize}

\subsection{Weierstrass Semigroup}  Here we present the relevant data on Weierstrass semigroups that will be used in Section 3.  For a more detailed discussion on the basics of this topic, see  \cite[chapter 6]{HKT}.

  For a rational point $P\in \mathcal{X}$,  the \textit{Weierstrass semigroup} of $\mathcal{X}$ at $P$ is defined by
$$
H(P):= \{n \in \mathbb{N}_{0} \mbox{ : } \exists f \in \mathbb{F}_{q}(\mathcal{X}) \mbox{ with } div_{\infty}(f) = n P \},
$$
and the  set $G(P) = \mathbb{N}_{0} \setminus H(P)$ is called \textit{Weierstrass gap set} of $P$. From the Weierstrass Gap Theorem,   $G(P)=\{\alpha_1,\ldots, \alpha_{g}\}$ and
$$1=\alpha_1 < \cdots < \alpha_{g} \leq 2g-1.$$ 
The semigroup $H(P)$ is called symmetric if $\alpha_g = 2g-1$. The curve $\mathcal{X}$ is called \emph{Castle curve} if $H(P)= \{ 0=m_1 < m_2 < \cdots \}$ is symmetric and $\# \mathcal{X}(\mathbb{F}_q)= m_2 q + 1$.
 For more details about Castle curves and their applications in coding theory, see \cite{MST}.

\begin{definition}\label{telescopic}
Let $(a_{1}, \ldots , a_m)$ be a sequence of positive integers whose greatest common divisor is $1$. Set $d_0=0$, and define $d_i := gcd (a_1, \ldots , a_i)$ and $A_i:= \{ \frac{a_1}{d_i}, \ldots , \frac{a_i}{d_i} \}$ for $i=1,\ldots,m$.  If $\frac{a_i}{d_i}$ lies in the semigroup generated by $A_i$ for $i=2, \ldots , m$, then the sequence $(a_1 , \ldots , a_m)$ is called \textit{telescopic}. A semigroup is called  telescopic if
it is generated by a telescopic sequence.
\end{definition}

For a semigroup $S$, the number of gaps and the largest gap of $S$ will be denoted by $g(S)$ and $l_g(S)$, respectively. The following result will be a significant factor  in  determining the semigroup $H(P_{\infty})$ of the curves $\mathcal{X}_{n,r}$.

\begin{lemma} [\cite{kirfel}, Lemma 6.5] \label{lemma telescopic}
If $S_m$ is  the semigroup generated by a telescopic sequence $(a_1 , \ldots, a_m)$, then
\begin{itemize}

\item $\displaystyle l_g(S_m) = d_{m-1} l_g(S_{m-1}) + (d_{m-1}-1)a_{m} = \sum_{i=1}^{m} (\frac{d_{i-1}}{d_i}-1)a_i$

\item $g(S_m)= d_{m-1} g(S_{m-1}) + (d_{m-1}-1)(a_{m}-1)/2 = (l_{g}(S_m) + 1)/2,$

\end{itemize}
where  $d_{0}=0$. In particular, telescopic semigroups are symmetric. 
\end{lemma}

\subsection{Automorphism Group} This  subsection  briefly recalls few facts related to  automorphism groups of algebraic curves over finite fields. For a more detailed discussion on  this subject, see  \cite{GK1}, \cite{GK}, \cite[chapter 11]{HKT}.

Let $\aut(\mathcal{X})$ be the automorphism group of $\mathcal{X}$   and $\mathbb{G} \subseteq \aut(\mathcal{X})$  be a finite subgroup. For a rational point $P \in \mathcal{X}$, the stabilizer of $P$ in $\mathbb{G}$, denoted by $\mathbb{G}_{P}$, is the subgroup of $\mathbb{G}$ consisting of all elements fixing $P$. For a non-negative integer $i$, the $i$-th ramification group of $\mathcal{X}$ at $P$ is denoted by $\mathbb{G}_{P}^{(i)}$ and defined by
\[\mathbb{G}_{P}^{(i)}=\{\alpha \in \mathbb{G}_{P} \mbox{ : } \mbox{ } v_P(\alpha(t)-t)\geq i+1 \}\;,\]
where  $t$ is a local parameter at $P$. Here $\mathbb{G}_{P}^{(0)}=\mathbb{G}_{P}$ and $\mathbb{G}_{P}^{(1)}$ is the unique Sylow $p$-subgroup of $\mathbb{G}_{P}$. Moreover, $\mathbb{G}_{P}^{(1)}$ has a cyclic complement $H$ in $\mathbb{G}_{P}$, i.e.,
\begin{equation}\label{Gp}
\mathbb{G}_{P}=\mathbb{G}_{P}^{(1)}\rtimes H
\end{equation}
where $H$ is a cyclic group of order coprime to $p$.

The following result will be essential to determine the full group $\aut(\mathcal{X}_{n,r})$ of automorphisms  of the curves $\mathcal{X}_{n,r}$.

\medskip

\begin{theorem} \label{teorema korchmaros} {\rm (\cite{HKT}, Theorem 11.140)} Let $\mathcal{X}$ be a curve of genus $g \geq 2$ over $\mathbb{F}_{q}$, where $q$ is a prime power, and let $\mathbb{G}$ be an automorphism group of $\mathcal{X}$ such that $\mathcal{X}$ has a $\mathbb{F}_{q}$-rational point $P$ satisfying the condition $|\mathbb{G}_{P}^{(1)}| > 2g+1$. Then one of the following cases occurs:

\begin{enumerate}[\rm 1)]

\item $\mathbb{G} = \mathbb{G}_{P}$.

\item  $\mathcal{X}$ is birationally equivalent to one of the following curves:

\end{enumerate}

\begin{enumerate}[\rm (i)]
\item  the Hermitian curve $\mathbf{v}(Y^n + Y - X^{n + 1})$ with $n=q^t \geq 2$ and $g=\frac{1}{2}(n^2-n)$

\item  the DLS curve (the Deligne-Lusztig curve arising from the Suzuki group) $\mathbf{v}(X^{n_{0}}(X^n + X) - (Y^n + Y))$ with $p=2$, $q=n$, $n_{0}=2^r$, $r\geq 1$, $n=2n_{0}^2$ and $g=n_{0}(n-1)$

\item the DLR curve (the Deligne-Lusztig curve arising from the Ree group) $\mathbf{v}(Y^{n^2}- (1+(X^n - X)^{n-1})Y^n + (X^n - X)^{n-1}Y - X^n (X^n - X)^{n+3n_{0}}$ with $p=3$, $q=n$, $n_{0}= 3^r$, $n=3n_{0}^2$ and $g= \frac{3}{2}n_{0}(n-1)(n+n_{0}+1)$.
\end{enumerate}

\end{theorem}

\subsection{The curves $\mathcal{X}_{n,r}$} In this subsection, we define the principal object of this study, namely the curves  $\mathcal{X}_{n,r}$,  and recall some of their basic properties. For additional   details, see \cite{c2}. 

  Fix  integers $n\geq 2$ and $r \in \{\lceil \frac{n}{2} \rceil,\ldots,n-1\}$, with $\gcd(n,r)=1$. Consider  the polynomial

\begin{equation}\label{eq:f_definition}
f_{r}(x):=T_n\left(x^{1+q^{r}}\right) \mod (x^{q^n}-x),
\end{equation}
 where  $T_n(x)=x+x^q+\ldots +x^{q^{n-1}}$, and define the curve 

\begin{equation} \label{eq Xr curve}
\mathcal{X}_{n,r}: T_n(y)=f_r(x).
\end{equation}

It is easy to check that the polynomial $f_r$ satisfies  $f_r(a) \in \F_{q}$ for all $a\in \F_{q^n}$, and that if $n>2$, then $f$ can written as
\begin{equation}\label{poly-explicit}
f_r(x)  = \sum_{i=0}^{n-r-1} (x^{1+q^r})^{q^i}+\sum_{i=0}^{r-1} (x^{1+q^{n-r}})^{q^i}.
\end{equation}

\begin{proposition} [\cite{c2}, Proposition 3.1] The following holds:
\begin{enumerate}[\rm 1)]

\item The curve $\mathcal{X}_{n,r}$ has degree $d=q^{n-1}+q^{r-1}$, genus $g=q^r(q^{n-1}-1)/2$
and $N=q^{2n-1}+1$ $\F_{q^n}$-rational points.

\item In the projective closure of $\mathcal{X}_{n,r}$, the point $P_{\infty}=(0:1:0)$ is the only singular point.

\end{enumerate}
\end{proposition}

\section{The Weierstrass semigroup $H(P_{\infty})$}

In this section, we  consider the  point $P_{\infty}=(0:1:0) \in \mathcal{X}_{n,r}$ and determine its  Weierstrass semigroup $H(P_{\infty})$.

Let $F_{n,r}=\mathbb{F}_{q^n}(x,y)$ be the function field of the curve $\mathcal{X}_{n,r}$ over $\mathbb{F}_{q^n}$, and let $v_{P_{\infty}}$ be the discrete valuation at $P_{\infty}$.

\begin{remark}
\begin{enumerate}[\rm 1)]
\item By equation (\ref{eq Xr curve}), the functions $x,y \in F_{n,r}$  satisfy
\begin{equation} \label{eq3}
y^{q^n}-y=(x^{q^n}-x)f_r'(x)=(x^{q^n}-x)(x^{q^r}+x^{q^{n-r}}).
\end{equation}
\item The notation 
$$z_0:=y^{q^{n-r}}-x^{q^{n-r}+1} \in F_{n,r}$$
 will be used  in what follows.
\end{enumerate}
\end{remark}

\begin{lemma} \label{lemma functions} For the functions $x,y,z\in F_{n,r}$, where
$$z := z_0^{q^{2r-n}}-x^{q^r + 1} + x^{q^{2r-n} - 1}y,$$ 
we have that
\begin{enumerate}[\rm 1)]

\item $div_\infty(x)=q^{n-1}P_{\infty}$

\item $div_\infty(y)=(q^{n-1}+q^{r-1})P_{\infty}$

\item $div_\infty(z)=(q^{2r-1}+q^{n-r-1})P_{\infty}$.

\end{enumerate}
\end{lemma}

\begin{proof}
The proof of assertions 1 and 2 follows similarly to \cite[Lemma 3.6.]{c2}. To compute  $v_{P_{\infty}}(z)$, we  use
 the function $z^{q^n}$.  Making use of equation (\ref{eq3}), we have

$\begin{array}{ll}
  z^{q^n} & = z_0^{q^{2r}}-x^{q^{n+r} + q^n} + x^{q^{2r} - q^n}((x^{q^n}-x)(x^{q^r}+x^{q^{n-r}})+y) \\
   & = y^{q^{n+r}}-x^{q^{n+r}+q^{2r}} -x^{q^{n+r} + q^n} + x^{q^{2r} - q^n}((x^{q^n}-x)(x^{q^r}+x^{q^{n-r}})+y) \\
   & = ((x^{q^n}-x)(x^{q^r}+x^{q^{n-r}})+y)^{q^r} -x^{q^{n+r}+q^{2r}} -x^{q^{n+r} + q^n}  \\
   & \mbox{ } \mbox{ } + x^{q^{2r} - q^n}((x^{q^n}-x)(x^{q^r}+x^{q^{n-r}})+y) \\
   & = (x^{q^n+q^r}+x^{q^n+q^{n-r}}-x^{q^r+1}-x^{q^{n-r}+1}+y)^{q^r} -x^{q^{n+r}+q^{2r}} -x^{q^{n+r} + q^n}  \\
   & \mbox{ } \mbox{ } + x^{q^{2r} - q^n}( x^{q^n+q^r} + x^{q^n+q^{n-r}} - x^{q^r+1} - x^{q^{n-r}+1} +y) \\
   & = \cancel{x^{q^{n+r}+q^{2r}}}+ \cancel{x^{q^{n+r}+q^{n}}}- \cancel{x^{q^{2r}+q^r}}-x^{q^{n}+q^r}+y^{q^r} - \cancel{x^{q^{n+r}+q^{2r}}} - \cancel{x^{q^{n+r} + q^n}}  \\
   &\mbox{ } \mbox{ } + \cancel{x^{q^{2r}+q^r}} + x^{q^{2r}+q^{n-r}} - x^{q^{2r} - q^n+q^r+1} - x^{q^{2r} - q^n+q^{n-r}+1} + x^{q^{2r} - q^n}y. 
 \end{array}$
 
Thus $z^{q^n} = x^{q^{2r}+q^{n-r}} - x^{q^{2r} - q^n+q^r+1} - x^{q^{2r} - q^n+q^{n-r}+1} + x^{q^{2r} - q^n}y -x^{q^{n}+q^r}+y^{q^r}$.
And  triangle inequality gives us
$$
v_{P_{\infty}}(z^{q^n}) = v_{P_{\infty}}(x^{q^{2r}+q^{n-r}}) = -q^{n-1}(q^{2r}+q^{n-r}).
$$

Therefore, $div_{\infty}(z) = (q^{2r-1}+q^{n-r-1})P_{\infty}$.

\end{proof}

\begin{remark}
Note that, if $gdc(n,r)=1$, then there exists positive integers $\alpha$ and $\beta$ such that $(n-r)\alpha-\beta n=1$.
\end{remark}

\begin{proposition}\label{proposition functions}
Let $\alpha$ and $\beta$ be positive integers such that $(n-r)\alpha-\beta n=1$, and consider the following functions in $F_{n,r}$:
\begin{equation}
w:=\sum\limits_{i=0}^{\alpha-1}{z_0}^{q^{(n-r)i}}-\sum\limits_{i=1}^{\beta}{(x^{q^{n-r}+1}+x^{q^n+q^{n-r}})}^{q^{n(\beta-i)+1}}
\end{equation}
and
\begin{equation}
t:=x^{q^{2r-n+1}-q}w+{z}^q+x^{q^{2r-n+1}-q^{2r-n}-q+1}z.
\end{equation}
Then $div_{\infty}(w)=(q^n+q^{n-r})P_{\infty}$, and  $div_{\infty}(t)=(q^{2r}-q^n+q^r+1)P_{\infty}.$
\end{proposition}

\begin{proof} We begin by computing $w^{q^n}-w$.  Note that
$$\Big(\sum\limits_{i=1}^{\beta}{(x^{q^{n-r}+1}+x^{q^n+q^{n-r}})}^{q^{n(\beta-i)+1}}\Big)^{q^n}-\sum\limits_{i=1}^{\beta}{(x^{q^{n-r}+1}+x^{q^n+q^{n-r}})}^{q^{n(\beta-i)+1}}=$$
$$(x^{q^{n-r}+1}+x^{q^n+q^{n-r}})^{q^{n\beta+1}}-(x^{q^{n-r}+1}+x^{q^n+q^{n-r}})^{q}.$$

Since

$$\sum\limits_{i=0}^{\alpha-1}(z_0^{q^n}-z_0)^{q^{(n-r)i}}=(x^{q^{n-r}+1}-x^{q^{n}+q^{n-r}})-(x^{q^{n-r}+1}-x^{q^{n}+q^{n-r}})^{q^{n\beta+1}},$$
it follows that
\begin{equation}\label{ordemteta}
w^{q^n}=w + x^{q^{n-r}+1}-x^{q^{n}+q^{n-r}}+x^{q^{n-r+1}+q}-x^{q^{n+1}+q^{n-r+1}}.
\end{equation}

Therefore, by triangle inequality, $v_{P_{\infty}}(w^{q^n}) = v_{P_{\infty}}(x^{q^{n+1}+q^{n-r+1}})$. Thus $q^n v_{P_{\infty}}(w) =-(q^{n+1}+q^{n-r+1})q^{n-1}$, and then $div_{\infty}(w)=(q^{n}+q^{n-r})P_{\infty}$. Now to find  $div_{\infty}(t)$,  we see that

\medskip

\small{
$\begin{array}{cl}
t^{q^n}&=x^{q^{2r+1}-q^{n+1}}w^{q^n}+z^{q^{n+1}}+x^{q^{2r+1}-q^{2r}-q^{n+1}+q^n}z_{r}^{q^n}\\
& =x^{q^{2r+1}-q^{n+1}}(w+x^{q^{n-r}+1}-x^{q^{n}+q^{n-r}}+x^{q^{n-r+1}+q}-x^{q^{n+1}+q^{n-r+1}}) \\
& +(x^{q^{2r}+q^{n-r}} - x^{q^{2r} - q^n+q^r+1} - x^{q^{2r} - q^n+q^{n-r}+1} + x^{q^{2r} - q^n}y -x^{q^{n}+q^r}+y^{q^r})^q\\
& + x^{q^{2r+1}-q^{2r}-q^{n+1}+q^n}(x^{q^{2r}+q^{n-r}} - x^{q^{2r} - q^n+q^r+1} - x^{q^{2r} - q^n+q^{n-r}+1} + x^{q^{2r} - q^n}y\\
& -x^{q^{n}+q^r}+y^{q^r})\\
& = x^{q^{2r+1}-q^{n+1}} w + \cancel{x^{q^{2r+1}-q^{n+1}+q^{n-r}+1}} - \cancel{x^{q^{2r+1}-q^{n+1} +q^{n}+q^{n-r}}}+ \cancel{x^{q^{2r+1}-q^{n+1}+q^{n-r+1}+q}}\\
&\cancel{-x^{q^{2r+1}+q^{n-r+1}}} +\cancel{x^{q^{2r+1}+q^{n-r+1}}} - x^{q^{2r+1} - q^{n+1}+q^{r+1}+q} - \cancel{x^{q^{2r+1} - q^{n+1}+q^{n-r+1}+q}} \\
& + x^{q^{2r+1} - q^{n+1}}y^q -x^{q^{n+1}+q^{r+1}}+y^{q^{r+1}}+ \cancel{x^{q^{2r+1}-q^{n+1}+q^n+q^{n-r}}} - x^{q^{2r+1}-q^{n+1}+q^r+1} \\
&- \cancel{x^{q^{2r+1}-q^{n+1}+q^{n-r}+1}} + x^{q^{2r+1}-q^{n+1}}y -x^{q^{2r+1}-q^{2r}-q^{n+1}+q^n+q^{n}+q^r}+ x^{q^{2r+1}-q^{2r}-q^{n+1}+q^n}y^{q^r}.\\
\end{array}$
}

Then,

$t^{q^n}= x^{q^{2r+1}-q^{n+1}} w  - x^{q^{2r+1} - q^{n+1}+q^{r+1}+q} + x^{q^{2r+1} - q^{n+1}}y^q -x^{q^{n+1}+q^{r+1}}+y^{q^{r+1}}$
$$- x^{q^{2r+1}-q^{n+1}+q^r+1} + x^{q^{2r+1}-q^{n+1}}y -x^{q^{2r+1}-q^{2r}-q^{n+1}+q^n+q^{n}+q^r}+ x^{q^{2r+1}-q^{2r}-q^{n+1}+q^n}y^{q^r}.$$

Using triangle inequality, we have $v_{P_{\infty}}(t^{q^n})= v_{P_{\infty}}(x^{q^{2r+1} - q^{n+1}+q^{r+1}+q})$ from which it follows that $div_{\infty}(t)=(q^{2r}-q^n+q^r+1)P_{\infty}$.
\end{proof}

\begin{corollary} The curve $\mathcal{X}_{n,r}$ has a plane model given by
\begin{equation}\label{newmodel}
y^{q^{n-1}}+\cdots+y^q+y=x^{q^{n-r}+1}-x^{q^{n}+q^{n-r}}.
\end{equation}
\end{corollary}
\begin{proof} Consider the polynomial in $\mathbb{F}_{q^n}(x)[Z]$ given by
$$G(Z) =Z^q-Z-x^{q^{n-r}+1}+x^{q^{n}+q^{n-r}}-x^{q^{n-r+1}+q}+x^{q^{n+1}+q^{n-r+1}}.$$It follows from \eqref{ordemteta} that $G(T_n(w))=0$. Therefore, since $G(Z)$ may be written as $\prod\limits_{\lambda \in \mathbb{F}_{q}}(Z-(x^{q^{n-r}+1}-x^{q^{n}+q^{n-r}})-\lambda)$, we have that
$$T_n(w)=x^{q^{n-r}+1}-x^{q^{n}+q^{n-r}}+\lambda$$ for some $\lambda \in \mathbb{F}_{q}.$
Letting $a\in \mathbb{F}_{q^n}$ be such that $T_n(a)=\lambda$ and considering the function $\tilde{w}=w-a$, we obtain

$$T_n(\tilde{w})=x^{q^{n-r}+1}-x^{q^{n}+q^{n-r}}.$$

Since  \cite[Thorem 2.5]{c3} implies that \eqref{newmodel} is irreducible, the result follows.
\end{proof}

\begin{theorem}
Let $H(P_{\infty})$ be the Weierstrass semigroup at $P_{\infty}$. Then
$$
H(P_{\infty}) = \langle q^{n-1}, q^{n-1}+q^{r-1}, q^n+q^{n-r}, q^{2r-1}+q^{n-r-1}, q^{2r}-q^n+q^r+1 \rangle.
$$
Moreover, $H(P_{\infty})$ is a telescopic semigroup and, in particular, symmetric.
\end{theorem}

\begin{proof}
Let

 $\cdot$ $a_1 = q^{n-1}$;

 $\cdot$ $a_2 = q^{n-1}+q^{r-1}$;

 $\cdot$ $a_3 = q^n+q^{n-r}$;

 $\cdot$ $a_4 = q^{2r-1}+q^{n-r-1}$;

 $\cdot$ $a_5 = q^{2r}-q^n+q^r+1$; and

 $\cdot$ $S = \langle a_1, a_2, a_3, a_4, a_5 \rangle$.

 By Lemma \ref{lemma functions} and Proposition \ref{proposition functions}, it  follows that $S \subseteq H(P_{\infty})$.
 To prove the equality, we first show that $S$ is telescopic. Using the notations of Definition \ref{telescopic}, we have

 $\cdot$ $d_1 = q^{n-1}$ and $S_1 = \langle 1 \rangle$;

 $\cdot$ $d_2 = q^{r-1}$ and $S_2 = \langle q^{n-r}, q^{n-r}+1 \rangle$;

 $\cdot$ $d_3 = q^{n-r}$ and $S_3 = \langle q^{r-1}, q^{r-1} + q^{2r-n-1}, q^{r}+1 \rangle$;

 $\cdot$ $d_4 = q^{n-r-1}$ and $S_4 = \langle q^{r}, q^{r}+q^{2r-n}, q^{r+1}+q, q^{3r-n}+1 \rangle$;

 $\cdot$ $d_5=1$ and $S_5=S$.

 Note that

 \medskip

 $\cdot$ $\dfrac{a_{2}}{d_2} = q^{n-r}+1 \in S_1$;

 \medskip

 $\cdot$ $\dfrac{a_{3}}{d_3} = q^{r}+1 = (q^{2r-n}-1)q^{n-r}+(q^{n-r}+1) \in S_2$;

 \medskip

 $\cdot$ $\dfrac{a_4}{d_4} = q^{3r-n}+1 = (q^{2r-n+1}-q)q^{r-1}+(q^r+1) \in S_3$;

 \medskip

 $\cdot$ $\dfrac{a_5}{d_5}=q^{2r}-q^n+q^r+1 = (q^r - q^{n-r} - q^{2r-n} + 1)q^r + (q^{3r-n}+1) \in S_4$.

 \medskip

 Therefore, $S$ is telescopic and, by Lemma \ref{lemma telescopic}, symmetric. To complete the proof, we must show that $g(S)$ is equal to $g=q^r(q^{n-1}-1)/2$. Using the formula for the genus of a telescopic semigroup given in Lemma \ref{lemma telescopic},
$$
\displaystyle g(S_5)= (l_{g}(S_5) + 1)/2 = \left ( \sum_{i=1}^{5} (\frac{d_{i-1}}{d_i}-1)a_i + 1 \right ),
$$
and then  $g(S)=q^r(q^{n-1}-1)/2$, which completes the proof.
\end{proof}

\begin{corollary}
The curves $\mathcal{X}_{n,r}$ are Castle curves.
\end{corollary}

\section{Automorphism group of $\mathcal{X}_{n,r}$}

Let us consider the $q^{2n-1}$  affine points   $P:=(\delta,\mu) \in \mathcal{X}_{n,r}(\F_{q^n})$ and all the elements $\gamma \in \F_{q^n}$ such that
\begin{equation*}
\begin{cases}
\gamma ^{q-1}=1& \text{ if } n \text{ is odd  }\\
\gamma ^{q^2-1}=1& \text{ if } n \text{ is even  }.\\
\end{cases}
\end{equation*}

 Using that $\mathcal{X}_{n,r}$ is given by  $T_n(y)=f_r(x)$, where $f_r(x)  = \sum_{i=0}^{n-r-1} (x^{1+q^r})^{q^i}+\sum_{i=0}^{r-1} (x^{1+q^{n-r}})^{q^i}$, one can easily check that the  set $G$ of
maps on $F_{n,r}$, given by 
\begin{equation}\label{autom Xr}
	\alpha_{\gamma,P}:  (x,y) \longrightarrow  (\gamma  x+\delta, \gamma ^{1+q} y+(\delta^{q^{n-r}}+\delta^{q^r})\gamma  x+\mu),	 
	\end{equation}
is a subgroup of  $\aut(\mathcal{X}_{n,r})$,  whose elements fix $P_\infty$, i.e., 
$G\subseteq \aut_{P_\infty}(\mathcal{X}_{n,r})$.  Based on the above definition, note that the following subgroups of $G$:
$$ N=\{\alpha_{\gamma,P} \in G: \gamma=1\}  \text{ and } H=\{\alpha_{\gamma,P} \in G: P=(0,0)\}\cong \F_{q^{2-(n \mod 2)}}^{*} $$
have order $q^{2n-1}$ and  $q^{2-(n \mod 2)}-1$, repectively.

The aim of this section is to prove the following:
\begin{theorem}\label{main.aut} The group $G$ is the full group of automorphisms of  $\mathcal{X}_{n,r}$. Moreover, $N=\aut_{P_\infty}(\mathcal{X}_{n,r})^{(1)}$ and 
$$G=\aut_{P_\infty}(\mathcal{X}_{n,r})=N\rtimes H.$$ 

\end{theorem}

The proof of Theorem \ref{main.aut} will follow after some partial results.
We begin with the following:
\begin{lemma}\label{lema1}
$\aut_{P_\infty}(\mathcal{X}_{n,r})=G.$

\end{lemma}

\begin{proof} As mentioned before, the inclusion $G\subseteq \aut_{P_\infty}(\mathcal{X}_{n,r})$
can be easily verified.  

Now let  $\alpha\in \aut_{P_\infty}(\mathcal{X}_{n,r})$ be  an automorphism. Note that Lemma \ref{lemma functions} implies that $\{1,x\}$ and $\{1,x,y\}$ are bases for the Riemann-Roch spaces $\mathcal{L}(q^{n-1}P_\infty)$ and   $\mathcal{L}((q^{n-1}+q^{r-1})P_\infty)$, respectively.
Therefore,   $\alpha(x)\in \mathcal{L}(q^{n-1}P_\infty)$ and   $\alpha(y)\in \mathcal{L}((q^{n-1}+q^{r-1})P_\infty)$, and
\[\alpha(x)=cx+d\;  \text{ and }   \alpha(y)=ay+b_1x+b_0\;\]
for some  $a,b_0,b_1,c,d \in \overline{\F}_{q^n}$ with $ac\neq 0$. Since $\alpha$ is an automorphism, 
we must have 
$$F(cx+d,ay+b_1x+b_0)=\lambda F(x,y)$$
for some $\lambda \in \overline{\F}_{q^n}$ and $F(x,y):=T_n(y)-f_r(x)$. In particular, 

\begin{itemize}
\item $T_n(ay)=\lambda T_n(y)$ 
\item  $f_r(cx+d)-T_n(b_1x+b_0)=\lambda f_r(x)$.
\end{itemize}

Note that the first equality  gives $a=\lambda \in \F_{q} \backslash \{0\}$. Moreover,
using that $$f_r(x)  = \sum_{i=0}^{n-r-1} (x^{1+q^r})^{q^i}+\sum_{i=0}^{r-1} (x^{1+q^{n-r}})^{q^i}$$  and  comparing coefficients on the second equality   yields $f_r(d)=T_{n}(b_0)$ and

\begin{enumerate}[\rm(i)]
\item  $c^{{q^r}+1} = \lambda = c^{{q^{n-r}}+1}$
\item $(d^{q^r}+d^{q^{n-r}})c=b_1$
\item $(d+d^{q^{2n-2r}})c^{q^{n-r}}={b_1}^{q^{n-r}}.$
\end{enumerate}

Since $\lambda \in \F_{q}\backslash\{0\}$, condition (i) gives
$c^{q^{n-r}}=\frac{\lambda}{c}$, and then $c^{q^{n}}=\frac{\lambda}{c^{q^r}}=c.$
That is, $c\in \F_{q^n}$. Furthemore, since $c^{{q^r}+1} = c^{{q^{n-r}}+1}$ implies
$c^{{q^{2r-n}}-1}=1$, we have that   $c^{\gcd(q^{n}-1,q^{2r-n}-1)}=1$. However,  it follows from $\gcd(n,r)=1$ that

\begin{equation}\label{eq2} \gcd(q^{n}-1,q^{2r-n}-1)=
\begin{cases}
q-1&\text{ if } $n$ \text{ is odd }\\
q^2-1&\text{ if } $n$ \text{ is even },
\end{cases}
\end{equation}
which proves the  assertion about $c$. In particular, (i) gives $a=\lambda=c^{q+1}$.

 Using conditions (ii) and (iii), one can easily see that $d\in \F_{q^n}$ as well. Therefore $f_r(d) \in \F_{q}$, and then $f_r(d)=T_{n}(b_0)$ implies that $(d,b_0)$ is an $\F_{q^n}$-rational point on $\mathcal{X}_{n,r}$. This finishes the proof.

\end{proof}

\begin{lemma} \label{theorem order}  $\aut_{P_\infty}(\mathcal{X}_{n,r})^{(1)}=N$ and 
$\aut_{P_\infty}(\mathcal{X}_{n,r})=N\rtimes H$.

\end{lemma}

\begin{proof}

Since $|N|=q^{2n-1}$ and  $|G|=q^{2n-1}(q^{2-(n \mod 2)}-1)$, it follows from Subsection 2.2 that $N$ must be the unique Sylow $p$-subgroup of $G=\aut_{P_\infty}(\mathcal{X}_{n,r})$, i.e., $N=\aut_{P_\infty}^{(1)}(\mathcal{X}_{n,r})$. Also, from \eqref{Gp}, it  follows that the cyclic group $H \cong G/N$ is a complement in G, i.e., $G=N\rtimes H$.

\end{proof}
\noindent
{\bf Proof of Theorem \ref{main.aut}}

In view of the results in Lemmas \ref{lema1} and \ref{theorem order}, it suffices to prove that
$\aut(\mathcal{X}_{n,r})=\aut_{P_\infty}(\mathcal{X}_{n,r})$. In fact, first recall that $\mathcal{X}_{n,r}$
has genus $g= \frac{q^{r}(q^{n-1}-1)}{2}$. Thus, for  $n>2$, the curve $\mathcal{X}_{n,r}$ cannot be birrationally equivalent to any of the  curves  listed in Theorem \ref{teorema korchmaros}. Therefore, since 
  $$|\aut_{P_\infty}(\mathcal{X}_{n,r})^{(1)}| = q^{2n-1} > 2g+1 = q^{r}(q^{n-1}-1) + 1,$$
  if follows from Theorem \ref{teorema korchmaros} that $\aut(\mathcal{X}_{n,r})=\aut_{P_\infty}(\mathcal{X}_{n,r})$, as desired. \qed

\end{document}